\documentclass[preprint]{elsarticle}
\usepackage[colorlinks=true]{hyperref}
\usepackage{lineno}
\modulolinenumbers[1]
\journal{Journal of Applied Mathematical Modelling}

\usepackage{amssymb,float,latexsym,amsmath,amsthm,booktabs,graphicx,epstopdf,epsfig,caption,subcaption,color,bm}
\usepackage[english]{babel}
\usepackage[labelfont={bf}]{caption}
\usepackage[mathscr]{eucal}
\usepackage{enumitem}
\DeclareMathOperator{\Var}{Var}
\DeclareMathOperator{\Cov}{Cov}
\DeclareMathOperator{\diag}{diag}
\newtheorem{theorem}{Theorem}[section]
\newtheorem{lemma}{Lemma}[section]

\newtheorem{colo}{Corollary }[section]
\newtheorem{example}{Example}[section]
\newtheorem{remark}{Remark}[section]

\newtheorem{assumption}{Assumption}[section]

\biboptions{sort&compress}

\begin{document}

\begin{frontmatter}

\title{Nonparametric Asymptotic Distributions of Pianka's and MacArthur-Levins Measures}



\author[mymainaddress]{Tareq Alodat\corref{mycorrespondingauthor}}
\cortext[mycorrespondingauthor]{Corresponding author}
\ead{t.alodat@latrobe.edu.au}

\author[mysecondaryaddress]{M. T. Alodat}

\author[mymainaddress]{Dareen Omari}
\address[mymainaddress]{Department of Mathematics and Statistics, La Trobe University, Melbourne, 3086, Australia.}
\address[mysecondaryaddress]{Department of Statistics, Sultan Qaboos University, Muscat, 123, Oman.}

\begin{abstract}
This article studies the asymptotic behaviors of nonparametric estimators of two overlapping measures, namely Pianka's and MacArthur-Levins measures. The plug-in principle and the method of kernel density estimation are used to estimate such measures. The limiting theory of the functional of stochastic processes is used to study limiting behaviors of these estimators. It is shown that both limiting distributions are normal under suitable assumptions. The results are obtained in more general conditions on density functions and their kernel estimators. These conditions are suitable to deal with various applications. A small simulation study is also conducted to support the theoretical findings. Finally, a real data set has been analyzed for illustrative purposes.
\end{abstract}


\end{frontmatter}


\section{Introduction}\label{sec1}
In various phenomena, researchers usually need to measure the similarities or differences between different populations to support their hypotheses.
In the most common cases, the similarity of two samples is measured by comparing the two distributions  in terms of their sample means using available statistics such as the $t$ and $U$ statistics or other measures.
However, there are many situations in which researchers need to estimate the similarity between populations in terms of their density functions. The overlapping measures are commonly used in such situations.

The overlapping coefficient is defined as the intersection area of graphs of two or more probability density functions. 
It offers a simple method to identify the similarity between samples or populations that are usually described in terms of their distribution functions.
The overlapping coefficient was first introduced by~\cite{Weitzman70} and it is extensively used to improve the accuracy of interpretations and conclusions of data analysis.  In the literature, there are many overlap coefficients such as  Matusita’s \cite{matusita1955decision}, Morisita’s \cite{morisita1959measuring}, Weitzman’s \cite{Weitzman70}, MacArthur-Levins \cite{macarthur1967limiting,pianka1974niche} measures and others. A good review of such measures can be found in \cite{lu1989multivariate}. The range of the overlapping coefficient is between zero and one. Two distributions become more similar when the value of their overlapping measure approaches~one.

The overlapping measures are useful in many applications arising in different fields such as ecology, biogeography, niche structure, biology, clinical trials and others (see \cite{Mishra797,Mizuno2005,Abele79,Chao2000}).  
They can be also used to estimate distances between overlapping groups of clustering data, see \cite{goldberg}. In psychology context, overlapping concepts are considered as basis in defining various effect size measures such as Cohen's U index, McGraw and Wong's CL index and Huberty's $I$ degree of non-overlap index (see \cite{cohen88,wong,hobert}).
Such measures require some strong distributional assumptions to be satisfied such as the distributions are  
unimodal or belong to parametric distribution families. These assumptions ensure asymptotic statistical properties of the measures, but they might restrict the implementation of overlapping-based metrics.
Pastore and Calcagn{\`{\i}} \cite{pastore2019measuring} introduced an overlapping index which can be used to measure similarities or differences between any kind of distributions with no assumptions on the multimodality.
The overlapping coefficient was also used to measure the balance in propensity score methods and to select the appropriate model, see \cite{Belitser}.

Statistical inference for overlapping measures has been extensively studied for parametric models, in particular, for situations in which both populations follow the normal distribution. In a recent study by Montoya et al.~\cite{MONTOYA2019558}, a general closed form for the Weitzman overlapping measure based on two parameters was obtained, which facilitates the statistical inference on this measure. Madhuri and Mishra~\cite{Mishra1994} obtained estimates for overlap measures of normal populations with
equal means and provided approximated formulae to compute their bias and variance. Moreover, the construction of confidence intervals for such measures was addressed using Jackknife and Bootstrap methods in~\cite{Mishra2000}. 
The maximum-likelihood estimation (MLE) of the overlapping coefficient was obtained
in~\cite{inman89}. A hypothesis testing method, based on the overlapping index of two normal distributions with equal variance, was developed in \cite{Mishra86}. The confidence intervals were also constructed based on the non-central $t$- and $F$-distributions in~\cite{Reiser99}.
Matusita's population measure was estimated for bivariate normal
distributions with equal variance by ML and restricted ML estimates in~\cite{Minami2000}. 
However, this case is very limited in applications as populations' parameters  might be different. The asymptotic variances and bias for such estimates were also derived by the $\delta$-method.
Moreover, the plug-in estimate of Matusita's measure for multivariate normal populations, was evaluated by Monte Carlo methods in~\cite{lu1989multivariate}.
Also, expressions for Morisita's, MacArthur-Levins'and Pianka's measures were also derived. 
Inferential procedures of niche overlap measures were discussed for normal populations with heterogeneous variances in~\cite{Mishra797}.

In contrast, in many situations, the collected data have no reasonable parametric model. In this case,  
the unknown densities are estimated nonparametrically using kernel density estimates \cite{Belitser,Bradley2000,Schmidt006}. 
Such estimates can be easily computed and are available in various software. Because of their wide applications, overlapping measures have received considerable attention, in particular, the limiting distributions of their estimators. Yue and Clayton \cite{yue2005} introduced a new similarity measure and obtained its nonparametric MLE, which is consistent and asymptotically normal. N'drin and Hili
\cite{Hili2015}~proved that the kernel estimate of the density of the invariant distribution of a diffusion process was consistent and asymptotically normal. This estimate was also used in estimating the parameters of multidimensional diffusion processes using the minimum Hellinger distance method. Stine and Heyse
\cite{Stine2001} described properties of parametric and nonparametric estimates through several numerical studies. It was shown that the nonparametric method is suitable for skewed distributions. Ahmad \cite{ahmad1980} obtained a nonparametric estimate of the Morisita's measure and studied its asymptotic properties. Anderson et al. 
\cite{Anderson2012} introduced a methodology for nonparametric estimation of a polarization measure and described its inference and asymptotic distribution. The nonparametric asymptotic distributions of functional estimators were applied to sunspot image data and neural data of epilepsy patients in~\cite{moon2016nonparametric}.

In this paper, we are interested in two overlapping measures, namely, Pianka’s and MacArthur-Levin's measures. Despite these measures have gained an interest for a long time in various practical fields, their asymptotic theory has not been developed, to the best of our knowledge, under a general framework until this time. We develop their asymptotic distributions via the kernel estimation approach.

The rest of the article is organized as follows. In Section \ref{sec2}, we present some definitions and notations that will be used in the sequel. Section \ref{sec3}, introduces the kernel estimators of the overlapping measures and the main assumptions. Section \ref{sec4}, presents some auxiliary results that will be used to derive the main results. In section \ref{sec5}, we investigate the asymptotic distribution of Pianka's measure. In Section \ref{sec6}, we investigate the asymptotic distribution of the MacArthur-Levins measure. A Simulation study is given in Section~\ref{sec7}. An application to the breast cancer data is given in Section \ref{app}.
\section{Definitions and notations}\label{sec2}
In what follows $\mathcal{R}_{a}:=[-a,a],\ a>0$. Also, we denote by $L_{2}(\mathcal{R}_{a})$ the space of square integrable functions over $\mathcal{R}_{a}$, i.e. $$L_{2}(\mathcal{R}_{a}):=\left\{f:f:\mathcal{R}_{a}\to\mathbb{R},\int_{\mathcal{R}_{a}} f^{2}(x) dx<\infty\right\}.$$ For a density function $f(x)$ with $\mathcal{R}_{a}$ as its support, the constants $c_{f}$ and $C_{f}$ are defined as $c_{f}:=\inf\limits_{x\in\mathcal{R}_{a}}f(x)$ and $C_{f}:=\sup\limits_{x\in\mathcal{R}_{a}}f(x)$. We denote by $L_{2}^{\prime}(\mathcal{R}_{a})$ the daul space of $L_{2}(\mathcal{R}_{a})$.
Also, we denote by $f^{\prime}(x)$ and $f^{\prime\prime}(x)$ the first and second derivatives of $f(x)$, respectively. Also, we use the following notations: $I(r,s):=\int_{\mathcal{R}_{a}}f^{r}(x)g^{s}(x)dx$, $r,s\geq0$, and for a kernel density function $K(\cdot)$, $k_{ij}:=\int_{\mathcal{R}_{a}}u^{i}K^{j}(u)du$, for $i,j=0,1,2$.

Let $L_{2}(\mathcal{R}_{a})$ be endowed by the following inner product 
\[
\langle f,g \rangle:=\int_{\mathcal{R}_{a}}f(x)g(x)dx,
\]	
where $f,g\in L_{2}(\mathcal{R}_{a})$.
Then $(L_{2}(\mathcal{R}_{a}),\langle \cdot,\cdot\rangle)$ is an inner product space (see \cite{billingsley2013convergence}). The norm of a function $f\in L_{2}(\mathcal{R}_{a})$ is
\[
\Vert f\Vert:=\langle f,f \rangle^{\frac{1}{2}}=\left(\int_{\mathcal{R}_{a}}f^{2}(x)dx\right)^{\frac{1}{2}}.
\]

In this paper, we consider Pianka’s and MacArthur-Levin's overlapping measures that are defined as follows (see \cite{pianka1974niche,macarthur1967limiting}).


\[
\rho(f,g):=\dfrac{\langle f,g \rangle}{\Vert f\Vert\Vert g\Vert}\ \text{and}\ \Delta(f,g):=\dfrac{\langle f,g \rangle}{\Vert f\Vert^{2}}.
\]
where $f(\cdot)$ and $g(\cdot)$ are two probability density functions.

Further, it can be noticed that $\rho(f,g)$ represents the cosine of the angle between vectors $f$ and $g$ in $(L_2(\mathcal{R}_{a}), \langle \cdot,\cdot \rangle )$, so the smaller the angle the closer the two distributions. It can be also shown that $\rho(f,g)=\Delta(f,g)=1$, if and only if the two pdfs are equal almost surely. The overlapping measure $\Delta (f,g)$ might be greater or less than one. The case of $\rho(f,g)=0$ or $\Delta(f,g)=0$ indicates a complete dissimilarity between $f$ and $g$. These overlapping coefficients are used in many applications such as the comparison of income distributions, see \cite{Weitzman70}. Also, they are commonly used as a measure of distinctness of clusters (see \cite{sneath1977method}), and as the probability of failure in the stress-strength models of reliability analysis, see \cite{ichikawa1993meaning}.

In this paper, we focus on investigating the asymptotic behaviors of the plug-in kernel density estimators of $\rho(f,g)$ and $\Delta(f,g)$ using the method of convergence of functionals of stochastic processes.

\section{Kernel estimators and assumptions}\label{sec3}
Let $X_{1},\dots, X_{n}$ and $Y_{1},\dots, Y_{n}$ be two independent random samples drawn from the probability density functions $f(x)$ and $g(x)$, respectively. We assume that the two densities $f(x)$ and $g(x)$ have a common compact support. The kernel density estimators of $\rho(f,g)$ and $\Delta(f,g)$ are denoted by $\rho(f_{n},g_{n})$ and $\Delta(f_{n},g_{n})$, respectively,~with
\[
f_{n}(x)=\frac{1}{nh_{n}}\sum_{i=1}^{n}K\left(\dfrac{x-X_{i}}{h_{n}}\right)\ \text{and}\ g_{n}(x)=\frac{1}{nh_{n}}\sum_{i=1}^{n}K\left(\dfrac{y-Y_{i}}{h_{n}}\right),
\]
where $h_{n}$ is a sequence of positive numbers satisfies certain bandwidth assumptions and $K(\cdot)$ is a kernel density function. For common choices of $h_{n}$ and $K(\cdot)$, we refer the reader to \cite{wand1994kernel}. It is noteworthy that the estimator $\rho(f_{n},g_{n})$ has no closed form in general. 

The population  overlapping measures $\rho(f,g)$ and $\Delta(f,g)$ can be estimated by the plug-in principle, i.e., by replacing $f(x)$ and $g(x)$ by their kernel density estimators $f_{n}(x)$ and $g_{n}(x)$, respectively. Therefore, point estimators of $\rho(f,g)$ and $\Delta(f,g)$ are given, respectively, by
\[
\rho_{n}:=\rho(f_{n},g_{n})=\dfrac{\langle f_{n},g_{n} \rangle}{\Vert f_{n}\Vert\Vert g_{n}\Vert}\ \text{and}\ \Delta_{n}:=\Delta(f_{n},g_{n})=\dfrac{\langle f_{n},g_{n} \rangle}{\Vert f_{n}\Vert^{2}}.
\]  
In this paper, we are interested in the asymptotic behavior of
\[
\sqrt{nh_{n}}(\rho(f_{n},g_{n})-\rho(f,g))\ \text{and} \sqrt{nh_{n}}(\Delta(f_{n},g_{n})-\Delta(f,g)),\ \text{as}\ n\to\infty.
\]

As the properties of $\rho_{n}$ and $\Delta_{n}$ depend on the terms $f(\cdot),g(\cdot), K(\cdot)$ and $h_{n}$, we consider the following set of assumptions on these terms. These assumptions will be considered when we derive the asymptotic distribution of $\rho_{n}$ and $\Delta_{n}$.

\begin{assumption}[Populations' assumptions]\label{ass1} For two probability density functions $f(\cdot)$ and $g(\cdot)$ we assume the following:
	\begin{enumerate}[label=\rm(\roman*)]
		\item\label{ass_1} $f(x)$ and $g(x)$ are compactly supported, continuous on $\mathcal{R}_{a}$ for some $a>0$ and bounded away from $0$, i.e., $c_{f}, c_{g}>0$.
		\item\label{ass_2} The first derivatives of $f(x)$ and $g(x)$ are square integrable over $\mathcal{R}_{a}$, i.e., $f^{\prime},g^{\prime}\in L_{2}(\mathcal{R}_{a})$
		\item\label{ass_3} $f(x)$ and $g(x)$ admit third order continuously differentiable, bounded derivatives and Riemann integrable over $\mathcal{R}_{a}$.
	\end{enumerate}
\end{assumption}

\begin{remark}
	Let $F(\cdot)$ and $G(\cdot)$ be two absolutely continuous CDFs having the support $\mathbb{R}$ or $\mathbb{R}^+$  with pdfs $f(\cdot)$ and $g(\cdot)$, respectively. If $a,b\in\mathbb{R}$ are two support points with $a<b$, then the truncated versions $\bar{f}(x):=\frac{f(x)}{F(b)-F(a)}$ and $\bar{g}(x):=\frac{g(x)}{G(b)-G(a)},\ x \in[a,b]$, of $f(\cdot)$ and $g(\cdot)$ satisfy Assumption {\rm\ref{ass1}}.   
\end{remark}
Now we present some examples of density functions to show whether Assumption~{\rm\ref{ass1}} holds true.
\begin{example}
	Let $f(x)=\frac{3}{2}(x+\frac{1}{2})^{\frac{1}{2}}$ and $g(x)=\frac{3}{2}(\frac{1}{2}-x)^{\frac{1}{2}}$, $x\in\mathcal{R}_{\frac{1}{2}}$. Then,
	\begin{enumerate}[label=\rm(\roman*)]
		\item $f^{\prime}$ and $g^{\prime}$ are unbounded on and are not square integrable on $\mathcal{R}_{\frac{1}{2}}$
		\item $f^{\prime\prime}$ and $g^{\prime\prime}$ are neither bounded nor square integrable on $\mathcal{R}_{\frac{1}{2}}$
		\item Note that in Assumption {\rm\ref{ass1}\ref{ass_2}} it is assumed that the density functions $f$ and $g$ are bounded away from $0$. This condition is essential in the sequel analysis as it does not hold for some density functions. For example, $f(x)=\frac{3}{2}(x+\frac{1}{2})^{\frac{1}{2}},\ x\in \mathcal{R}_{\frac{1}{2}}$, does not satisfy this condition as $\int_{\mathcal{R}_{\frac{1}{2}}}(f^{\prime}(x))^{2}dx=\infty$.
	\end{enumerate} 
\end{example}

\begin{example}
	Let $f$ be the pdf of $N(\mu,\sigma^{2}),\ \mu\in\mathbb{R},\ \sigma>0$, and $g$ be the pdf of $Logistic(\theta_{1},\theta_{2})$, $\theta_{1}\in\mathbb{R},\ \theta_{2}>0$. Then the truncated versions of $f$ and $g$ on any interval $\mathcal{R}_{a},\ a>0$, satisfy Assumption {\rm\ref{ass1}}.
\end{example}

\begin{assumption}[Kernel's assumptions]\label{ass3}
	For the kernel function $K(\cdot)$, we assume the following conditions hold 
	\begin{enumerate}[label=\rm(\roman*)]
		\item $K(u),\ u\in\mathcal{R}_{a}$, is a pdf, compactly supported on $\mathcal{R}_{a}$, symmetric about zero and has a finite variance.
		\item $\lim\limits_{\vert y\vert\to\infty}\vert y\vert K(y)=0$.
	\end{enumerate} 
\end{assumption}

\begin{remark}\label{R2}
	In {\rm\cite{silverman1986density}} and {\rm\cite{li2007nonparametric}} it was assumed that the kernel function $K(\cdot)$ satisfies the condition $\int_{\mathcal{R}_{a}} K^{2+\delta}(u)du < \infty$, for $\delta >0$. Note that as $K(\cdot)$ is bounded on $\mathcal{R}_{a}$, then $K(\cdot)$ satisfies this condition in Assumption {\rm\ref{ass3}}. 
\end{remark}

\begin{assumption}[Bandwidth assumptions]\label{ass4}
	For each sample of size $n$, we assume that the sequence $\left\{h_{n}\right\}$ satisfies the following conditions
	\begin{enumerate}[label=\rm(\roman*)]
		\item \label{hi}$h_{n},nh_{n}^{3}\to 0$, as $n\to\infty$.
		\item $nh_n\to\infty$, as $n\to\infty$.
	\end{enumerate}
\end{assumption}
\begin{remark} Note that
	\begin{enumerate}[label=\rm(\roman*)]
		\item the conditions on the bandwidth $h_{n}$ in Assumption {\rm\ref{ass4}} are met by various real valued sequences. For example, it can be shown that $h_{n}=n^{-\alpha}$, $\alpha\in(\frac{1}{3},1)$, $h_{n}=n^{-\alpha}(\log n)^{\frac{1}{2}}$, $h_n=n^{-\alpha}(\log n)^{-1},\ \alpha\in[\frac{1}{3},1]$, and $h_{n}=n^{-2/3}(\log n)^{1/3}$ satisfy these conditions.
		\item In {\rm\cite{silverman1986density}} the optimal bandwidth selection of $h_{n}=cn^{-\frac{1}{5}}$, where $c$ is a constant depends on $f$ and $K$, does not satisfy the required conditions in Assumption {\rm\ref{ass4}}.
	\end{enumerate}
\end{remark}

In this article, we restrict our attention to the case of density functions $f$ and $g$ that are compactly supported on the compact support $\mathcal{R}_{a}$. This direction allows us to adopt the similarity in a wide range of applied problems.
\section{Auxiliary results }\label{sec4}
This section presents some known results from the asymptotic theory of nonparametric functional estimation which will be used to derive the main results of this paper. The general results in this section can be found in \cite{li2007nonparametric} and \cite{prakasa1983}.

Now we reformulate Theorem 1.3 in \cite{li2007nonparametric} based on our assumptions. 
\begin{theorem}\label{A}
	Let $X_{1},\dots,X_{n}$ be i.i.d. random variables with a probability density function $f(\cdot)$ satisfies Assumption {\rm\ref{ass1}}. Let $K(\cdot)$ and $h_{n}$ satisfy Assumptions {\rm\ref{ass3}} and {\rm\ref{ass4}}. If $nh^{7}_{n}\to 0$, then for any finite set of points $x_{1},\dots,x_{r}$, it holds~that
	\[
	\sqrt{nh_{n}}\begin{pmatrix}
		f_{n}(x_{1})-f(x_{1})-\frac{k_{21}}{2}h_{n}^{2}f^{\prime\prime}(x_{1})\\
		\vdots\\
		f_{n}(x_{r})-f(x_{r})-\frac{k_{21}}{2}h_{n}^{2}f^{\prime\prime}(x_{r})
	\end{pmatrix}\xrightarrow{d} N(\bm 0,k_{02}\diag(f(x_{1}),\dots,f(x_{r}))).
	\]
\end{theorem}
Notice that as the function $f^{\prime\prime}(\cdot)$ is bounded and 
\[
\sqrt{nh_{n}}h_{n}^{2}=\dfrac{nh_{n}^{3}}{\sqrt{nh_{n}}}\to 0,\quad n\to\infty,
\]
then, when $n\to\infty$, the term $\sqrt{nh_{n}}h_{n}^{2}f^{\prime\prime}(\cdot)$ in the above theorem converges to $0$ and it can be dropped without affecting the validity of the limiting distribution. Also the bandwidth in Assumption \ref{ass4} implies the required condition $nh_{n}^{7}\to0$ as $n\to\infty$ in Theorem \ref{A}.

The following result is a particular case of Theorem~3.1.4 in \cite{prakasa1983}.
\begin{theorem}[\cite{prakasa1983}, p.189]\label{Pra}
	Suppose that $K(\cdot)$ and the sequence $\{h_{n}\}$ satisfy Assumptions {\rm\ref{ass3}} and {\rm\ref{ass4}}. Further, the density $f(x)$ has continuous partial derivatives of the third order in a neighborhood of $x$. Then the bias $B_{n,f}(x)=\mathbb{E}f_{n}(x)-f(x)$ satisfies 
	\begin{equation}\label{bias}
		\lim_{n\to\infty}h_{n}^{-2}B_{n,f}(x)=\frac{1}{2}f^{\prime\prime}(x)k_{21}.
	\end{equation}
\end{theorem}
By Assumption {\rm\ref{ass3}} the integral on the RHS of ({\rm\ref{bias}}) is absolutely convergent, $K(\cdot)$ is bounded and $\int_{\mathcal{R}_{a}}y^{2}K(y)dx<\infty$. It follows (see \cite{prakasa1983}) that   
\begin{align}\label{44.44}
	nh_{n}\int_{\mathcal{R}_{a}}\Var(f_{n}(x))dx\leq \int_{\mathcal{R}_{a}}K^{2}(x)dx=k_{02}<\infty.
\end{align}

\section{Lemmas and main results}\label{sec5}
In this section, we derive the limiting distribution of $\rho(f_{n},g_{n})$, when $n\to\infty$. Also, we prove some lemmas that will be used to obtain the main results.
\begin{lemma}\label{lem1}
	Let $f_{n}(x)$ and $g_{n}(x)$ be kernel density estimators of $f(x)$ and $g(x)$, respectively. If Assumptions {\rm\ref{ass1}-\rm\ref{ass4}} hold, then
	\begin{align*}
		I_{n}(r,s):=\int_{\mathcal{R}_{a}}f_{n}^{r}(x)g_{n}^{s}(x)dx\xrightarrow{p} I(r,s),
	\end{align*}
	 for $\ r,s\in\left\{0,1,2,3,1/2,3/2,5/2\right\}$.
\end{lemma}

\begin{proof}
	Here we present the proof for the case $r=1/2$ and $s=5/2$, while other cases can be proved analogously. Now we show that $I_{n}\left(\frac{1}{2},\frac{5}{2}\right)\xrightarrow{p} I\left(\frac{1}{2},\frac{5}{2}\right)$, i.e.
	\[
	\int_{\mathcal{R}_{a}}f_{n}^{\frac{1}{2}}(x)g_{n}^{\frac{5}{2}}(x)dx\xrightarrow{p}\int_{\mathcal{R}_{a}}f^{\frac{1}{2}}(x)g^{\frac{5}{2}}(x)dx.
	\]
	It is sufficient to show that 
	\[
	\mathcal{J}_{n}:=\mathbb{E}\bigg\vert\int_{\mathcal{R}_{a}}f_{n}^{\frac{1}{2}}(x)g_{n}^{\frac{5}{2}}(x)dx-\int_{\mathcal{R}_{a}}f^{\frac{1}{2}}(x)g^{\frac{5}{2}}(x)dx\bigg\vert\to 0.
	\]
	Using the identity $\vert\sqrt{a}-\sqrt{b}\vert\leq \frac{\vert a-b\vert}{\sqrt{b}},\ a,b>0,$ we can estimate $\mathcal{J}_{n}$ as 
	\begin{align*}
		\mathcal{J}_{n}&\leq\mathbb{E}\bigg\vert\int_{\mathcal{R}_{a}}\left(\dfrac{f_{n}(x)g_{n}^{5}(x)dx-f(x)g^{5}(x)}{f^{\frac{1}{2}}(x)g^{\frac{5}{2}}(x)}\right)dx\bigg\vert,\\
		&\leq 2C\mathbb{E}\left(\int_{\mathcal{R}_{a}}f_{n}(x)\vert g_{n}^{5}(x)-g^{5}(x)\vert dx+\int_{\mathcal{R}_{a}}g^{5}(x)\vert f_{n}(x)-f(x)\vert dx\right),\\
		&\leq 2C\int_{\mathcal{R}_{a}}\mathbb{E}f_{n}(x)\mathbb{E}\vert g_{n}^{5}(x)-g^{5}(x)\vert dx+2C\int_{\mathcal{R}_{a}}g^{5}(x)\mathbb{E}\vert f_{n}(x)-f(x)\vert dx,\\
		&=: 2C(\mathcal{J}_{1n}+\mathcal{J}_{2n}),
	\end{align*}
	where
	\[
	\mathcal{J}_{1n}=\int_{\mathcal{R}_{a}}\mathbb{E}f_{n}(x)\mathbb{E}\vert g_{n}^{5}(x)-g^{5}(x)\vert dx
	\]
	and
	\[
	\mathcal{J}_{2n}=\int_{\mathcal{R}_{a}}g^{5}(x)\mathbb{E}\vert f_{n}(x)-f(x)\vert dx.
	\]
	Now we consider the term $\mathcal{J}_{2n}$. 
	
	Notice that $\int_{\mathcal{R}_{a}}\mathbb{E}\vert f_{n}(x)-f(x)\vert dx\to 0$ (see \cite{ahmad1980,Parzen1962}) then
	\[
	\mathcal{J}_{2n}\leq 2C\sup_{x\in\mathcal{R}_{a}}g^{5}(x)\int_{\mathcal{R}_{a}}\mathbb{E}\vert f_{n}(x)-f(x)\vert dx\to 0,\ \quad n\to\infty.
	\]
Now we consider the term $\mathcal{J}_{1n}$. Notice that
the term $\mathbb{E}f_{n}(x)$ can be estimated~as
\[
\mathbb{E}f_{n}(x)=\dfrac{1}{h_{n}}\mathbb{E}K\left(\frac{x-X_{1}}{h_{n}}\right)=\int_{\mathcal{R}_{a}}K(v)f(x-h_{n}v)dv\leq C\int_{\mathbb{R}}K(v)dv<\infty.
\]
As $f$ and $K$ are bounded functions, it holds, for all $n$, that $\mathbb{E}f_{n}(x)\leq C<\infty$.
Then to find an upper bound for $\mathcal{J}_{1n}$, it is enough to investigate the term $\mathbb{E}\vert g_{n}^{5}(x)-g^{5}(x)\vert$. 

By applying the factorization of $a^{n}-b^{n}$ and the Schwarz's inequality, we can bound $\mathbb{E}\vert g_{n}^{5}(x)-g^{5}(x)\vert$ as follows:
\begin{align*}
	&\mathbb{E}\vert g_{n}^{5}(x)-g^{5}(x)\vert=\\
	&\mathbb{E}\vert g_{n}(x)-g(x)\vert\big( g_{n}^{4}(x)+g_{n}^{3}(x)g(x)+g_{n}^{2}(x)g^{2}(x)+g_{n}(x)g(x)^{3}+g^{4}(x)\big),\\
	&\leq \sqrt{\mathbb{E}(g_{n}(x)-g(x))^{2}}\\
	&\times\sqrt{(g_{n}^{4}(x)+g_{n}^{3}(x)g(x)+g_{n}^{2}(x)g^{2}(x)+g_{n}(x)g(x)^{3}+g^{4}(x))^{2}}.
\end{align*}
Note that the term $\mathbb{E}(g_{n}(x)-g(x))^{2}\to 0$ (see \cite{Parzen1962}) and the right-hand side of
\[
\mathbb{E}(g_{n}(x)-g(x))^{2}\leq\mathbb{E}g_{n}^{2}(x)+g^{2}(x) 
\]
is uniformly bounded. Also, the term $\mathbb{E}(g_{n}^{4}(x)+g_{n}^{3}(x)g(x)+g_{n}^{2}(x)g^{2}(x)+g_{n}(x)g(x)^{3}+g^{4}(x))^{2}$ is bounded provided that $\mathbb{E} g_{n}^{m}(x)$ is bounded for $m\in\left\{1,2,\dots,8\right\}$. So, it is sufficient to show that $\mathbb{E} g_{n}^{m}(x)$ is bounded for all possible values of $m$.

Now we show that $\mathbb{E}g_{n}^{m}(x)$ is uniformly bounded function in $x$ and $n$. 
\begin{align*}
	\mathbb{E}g_{n}^{m}(x)&=(nh_{n})^{-m}\sum_{i_1=1}^{n}\cdots\sum_{i_m=1}^{n}\mathbb{E}\left(\prod_{j=1}^{m}K\left(\frac{x-X_{i_{j}}}{h_{n}}\right)\right),\\
	&=(nh_{n})^{-m}\sum_{s=1}^{m}{n \choose s}\sum^{*}\prod_{j=1}^{s}\mathbb{E}K^{i_{j}}\left(\frac{x-X_{i_{j}}}{h_{n}}\right),
\end{align*}
where the sum $\sum\limits^{*}$ extends over all indices $\left\{i_{1},\dots,i_{s}\right\}$ such that $\sum_{j=1}^{s}i_{j}=m$ and $i_{j}>0$ for some $j\in\left\{1,2,\dots,s\right\}$.

Using the above notation, we arrive at
\begin{align*}
	\mathbb{E}g_{n}^{m}(x)&=\sum_{s=1}^{m}\dfrac{{n \choose s}}{n^{m}h_{n}^{m}}\sum^{*}\prod_{j=1}^{s}h_{n}\int_{\mathcal{R}_{a}} K^{i_{j}}\left(v\right)f(x-h_{n}v)dv,\\
	&=\sum_{s=1}^{m}\dfrac{{n \choose s}}{n^{m}h_{n}^{m-s}}\sum^{*}\prod_{j=1}^{s}h_{n}\int_{\mathcal{R}_{a}} K^{i_{j}}\left(v\right)f(x-h_{n}v)dv.
\end{align*}
Note that for sufficiently large $n$, it holds that 
\[
\dfrac{{n \choose s}}{n^{m}h_{n}^{m-s}}=\dfrac{n(n-1)\cdots(n-s+1)}{n^{s}(nh_{n})^{m-s}}\leq 1.
\]
Hence, $\mathbb{E}g_{n}^{m}(x)$ is uniformly bounded in $x$ and $n$, which completes the proof.
\end{proof}

Let two stochastic processes $\eta_{1n}(x)$ and $\eta_{2n}(x)$, $x\in\mathcal{R}_{a}$, be defined as
\begin{align}\label{eta}	\eta_{1n}(x):=\dfrac{\sqrt{nh_{n}}(f_{n}(x)-f(x))}{\sqrt{k_{02}f(x)}}\ \text{and}\ \eta_{2n}(x):=\dfrac{\sqrt{nh_{n}}(g_{n}(x)-g(x))}{\sqrt{k_{02}g(x)}}.
\end{align}
Also, we define the following terms
\[
e_{n}^{(1)}=\dfrac{k_{02}}{\sqrt{nh_{n}}}\int_{\mathcal{R}_{a}}\sqrt{f(x)g(x)}\eta_{1n}(x)\eta_{2n}(x)dx,
\]
\[
e_{n}^{(2)}=\dfrac{\sqrt{k_{02}}}{\sqrt{nh_{n}}}\int_{\mathcal{R}_{a}}f(x)\eta_{1n}^{2}(x)dx\ \text{and}\ e_{n}^{(3)}=\dfrac{\sqrt{k_{02}}}{\sqrt{nh_{n}}}\int_{\mathcal{R}_{a}}g(x)\eta_{2n}^{2}(x)dx.
\]
Then we prove the following Lemmas.
\begin{lemma}\label{NA}
	If Assumptions {\rm\ref{ass1}}-{\rm\ref{ass4}} hold, then the term $\mathbb{E}\eta_{1n}^{2}(x)+\mathbb{E}\eta_{2n}^{2}(x)$ is bounded as a sequence of functions of $n$.
\end{lemma}
\begin{proof}
	Note that it is sufficient to show that $\mathbb{E}\eta_{1n}^{2}(x)$ is bounded as a sequence of functions of $n$. To this end, 
	\begin{align*}
		\mathbb{E}\eta_{1n}^{2}(x)&=\dfrac{nh_{n}}{k_{02}f(x)}\mathbb{E}(f_{n}(x)-f(x))^{2},\\
		&\leq \dfrac{2nh_{n}}{k_{02}f(x)}\left\{\mathbb{E}(f_{n}(x)-\mathbb{E}f_{n}(x))^{2}+(\mathbb{E}f_{n}(x)-f(x))^{2}\right\},\\
		&\leq\dfrac{2nh_{n}}{k_{02}f(x)}\mathbb{E}\bigg\vert\dfrac{1}{nh_{n}}\left(\sum_{i=1}^{n}\left(K\left(\dfrac{x-X_{i}}{h_{n}}\right)-\mathbb{E}K\left(\dfrac{x-X_{i}}{h_{n}}\right)\right)\right)\bigg\vert^{2}\\
		&+\dfrac{2nh_{n}}{k_{02}f(x)}B_{n,f}^{2}(x).
	\end{align*} 
	To simplify the calculation, let us define the random variable $S_{n}$ as
	\[
	S_{n}=\sum_{i=1}^{n}\left(K\left(\dfrac{x-X_{i}}{h_{n}}\right)-\mathbb{E}K\left(\dfrac{x-X_{i}}{h_{n}}\right)\right)=\sum_{i=1}^{n}Z_{i},\ {\rm say}.
	\]
	Then the term $\mathbb{E}\eta_{1n}^{2}(x)$ can be estimated as \begin{align}\label{S}
		\mathbb{E}\eta_{1n}^{2}(x)&\leq\dfrac{2nh_{n}}{k_{02}f(x)}\mathbb{E}\left(\dfrac{S_{n}}{n}\right)^{2}+\dfrac{2nh_{n}}{k_{02}f(x)}B_{n,f}^{2}(x).
	\end{align} 
	As $Z_{i}$'s are i.i.d. random variables with $\mathbb{E}Z_{1}=0$ and $\sigma^{2}=\Var(Z_{1})<\infty$. Then,
	\[
	\mathbb{E}\left(\dfrac{S_{n}}{n}\right)^{2}=\dfrac{\Var(Z_{n})}{n}=\dfrac{\sigma^{2}}{n}.
	\]
	Let $\mathcal{R}_{a}^{\prime}:=[\frac{x-a}{h_{n}},\frac{x+a}{h_{n}}]$. Then $\sigma^{2}$ can be written as
	\begin{align*}
		\sigma^{2}&=\int_{\mathcal{R}_{a}}\left[K\left(\dfrac{x-X_{1}}{h_{n}}\right)-\mathbb{E}K\left(\dfrac{x-x_{1}}{h_{n}}\right)\right]^{2}f(x_{1})dx_{1},\\
		&=h_{n}\int_{\mathcal{R}_{a}^{\prime}}\left[K(v)-\mathbb{E}K\left(\dfrac{x-X_{1}}{h_{n}}\right)\right]^{2}f(x-h_{n}v)dv,\\
		&=h_{n}\int_{\mathcal{R}_{a}^{\prime}}\left[K(v)-h_{n}\int_{\mathcal{R}_{a}^{\prime}}K(w)f(x-h_{n}w)dw\right]^{2}f(x-h_{n}v)dv,\\
		&= h_{n}H_{n}(x).
	\end{align*}
	By Assumptions~\ref{ass1}(i), \ref{ass3}(i) and \ref{ass4}(i), $H_{n}(x)$ is bounded by an integrable function. Hence, when $n\to\infty$, (\ref{S}) becomes
	\begin{align*}
		\mathbb{E}\eta_{1n}^{2}(x)&\leq\dfrac{2n}{k_{02}f(x)h_{n}}h_{n}\dfrac{H_{n}(x)}{n}+\dfrac{2nh_{n}}{k_{02}f(x)}B_{n,f}^{2}(x),\\
		&\leq\dfrac{2H_{n}(x)}{k_{02}f(x)}+\dfrac{2nh_{n}}{k_{02}f(x)}B_{n,f}^{2}(x),\\
		&=\bar{C}+\dfrac{2nh_{n}}{k_{02}f(x)}B_{n,f}^{2}(x),
	\end{align*} 
	for a positive constant $\bar{C}$.
To show that the term $\frac{2nh_{n}}{k_{02}f(x)}B_{n,f}^{2}(x)$ is bounded, we write it as
\[
\frac{2nh_{n}}{k_{02}f(x)}B_{n,f}^{2}(x)=\frac{2nh_{n}^{3}}{k_{02}f(x)}\vert B_{n,f}(x)\vert\dfrac{\vert B_{n,f}(x)\vert}{h_{n}^{2}}.
\]
By Assumption \ref{ass1}\ref{ass_1} the term $\frac{2nh_{n}^{3}}{k_{02}f(x)}\to 0$ when $n\to\infty$. Also, notice that
\[
B_{n,f}(x)=\mathbb{E}f_{n}(x)-f(x)=\int_{\mathcal{R}_{a}^{\prime}}K(v)f(x-h_{n}v)dv-f(x).
\]
Hence,
\begin{align*}
	\big\vert B_{n,f}(x)\big\vert\leq C_{f}\int_{\mathcal{R}_{a}^{\prime}}K(v)dv+f(x)=C+f(x).
\end{align*}
By Theorem \ref{Pra} it follows that 
\[
\int_{\mathcal{R}_{a}}f^{\prime\prime}(x)y^{2}K(y)dy=f^{\prime\prime}(x)k_{21}.
\]
By Assumption \ref{ass1}\ref{ass_3} $f^{\prime\prime}(x)$ is bounded, which guarantees that (similarly for $B_{n,g}(x)$)
\begin{align}\label{bound-f}
	\lim_{n\to\infty}h_{n}^{-2}\vert B_{n,f}(x)\vert=\dfrac{1}{2}\vert f^{\prime\prime}(x)\vert k_{21}<\infty,
\end{align}
and
\begin{align}\label{bound-g}
	\lim_{n\to\infty}h_{n}^{-2}\vert B_{n,g}(x)\vert=\dfrac{1}{2}\vert g^{\prime\prime}(x)\vert k_{21}<\infty.
\end{align}
By combining the above results, we conclude that $\frac{2nh_{n}}{k_{02}f(x)}B_{n,f}^{2}(x)\to0$, when $n\to\infty$, which completes the proof.
\end{proof}

\begin{lemma}\label{lem2}
	Let $f_{n}(x)$ and $g_{n}(x)$ be kernel density estimators of $f(x)$ and $g(x)$, respectively. If Assumptions {\rm\ref{ass1}}-{\rm\ref{ass4}} hold, then $\bm e_{n} :=\left(e_{n}^{(1)},e_{n}^{(2)},e_{n}^{(3)}\right)^{T}\xrightarrow{P}\bm 0$ when $n\to\infty$.
\end{lemma}
\begin{proof}
	Let $\varepsilon>0$. Then for $\bm a\in\mathbb{R}^{3}\setminus \left\{\bm0\right\}$, it holds that
	\begin{align*}
		\mathbb{P}\left(\vert \bm a^{T}\bm e_{n}\vert>\varepsilon\right)&\leq\sum_{i=1}^{3}\mathbb{P}\left(\vert a_{i}e_{n}^{(i)}\vert>\frac{\varepsilon}{3}\right).	
	\end{align*}
	For the term $e_{n}^{(1)}$, 
	\begin{align}\label{E}
		\mathbb{P}\left(\vert a_{1}e_{n}^{(1)}\vert>\frac{\varepsilon}{3}\right)&=\mathbb{P}\left(\frac{\vert a_{1}\vert k_{02}}{\sqrt{nh_{n}}}\bigg\vert\int_{\mathcal{R}_{a}}\sqrt{f(x)g(x)}\eta_{1n}(x)\eta_{2n}(x)dx\bigg\vert>\frac{\varepsilon}{3}\right),\notag\\
		&\leq \frac{3\vert a_{1}\vert k_{02}}{\varepsilon\sqrt{nh_{n}}}\mathbb{E}\bigg\vert\int_{\mathcal{R}_{a}}\sqrt{f(x)g(x)}\eta_{1n}(x)\eta_{2n}(x)dx\bigg\vert.
	\end{align}
	By Schwarz's and Markov's inequalities, then (\ref{E}) becomes
	\begin{align*}
		\mathbb{P}\left(\vert a_{1}e_{n}^{(1)}\vert>\frac{\varepsilon}{3}\right)&\leq \frac{3\vert a_{1}\vert k_{02}}{\varepsilon\sqrt{nh_{n}}}\int_{\mathcal{R}_{a}}\sqrt{f(x)g(x)}\mathbb{E}\vert\eta_{1n}(x)\eta_{2n}(x)\vert dx,\\
		&\leq\frac{3\vert a_{1}\vert k_{02}}{\varepsilon\sqrt{nh_{n}}}\int_{\mathcal{R}_{a}}\sqrt{f(x)g(x)}\sqrt{\mathbb{E}\eta_{1n}^{2}(x)}\sqrt{\mathbb{E}\eta_{2n}^{2}(x)}dx,\\
		&\leq \frac{C}{\varepsilon\sqrt{nh_{n}}},
	\end{align*}
	where $0<C<\infty$. So by Assumption \ref{ass4}, $e_{n}^{(1)}\xrightarrow{p} 0$.
	The other two terms $e_{n}^{(2)}$ and $e_{n}^{(3)}$ can be treated similarly. Since $\bm a$ is an arbitrary, then by combining the above results for $\bm e_{n}$, we conclude that $\bm e_{n}\xrightarrow{p} \bm0$.
\end{proof}

Now we state the main result of this paper.
\begin{theorem}\label{theo1}
	Let $f(\cdot)$ and $g(\cdot)$ be two probability density functions satisfying Assumption {\rm\ref{ass1}}. Let $K(\cdot)$ and $h_{n}$ satisfy Assumptions {\rm\ref{ass3}} and {\rm\ref{ass4}}, respectively. Then, when $n\to\infty$, we~have
	\[
	\sqrt{nh_{n}}(\rho(f_{n},g_{n})-\rho(f,g))\xrightarrow{d} N\left(0,\sigma_{f,g}^{2}\right),
	\]
	where
	\[
\sigma_{f,g}^{2}=\dfrac{k_{02}(A(f,g)-2B(f,g)+C(f,g)-2D(f,g))}{E(f,g)},
	\]
	with
	\begin{align*}
		A(f,g)&=I(0,3)I^{2}(1,1)I^{2}(2,0),\ B(f,g)=I\left(1/2,5/2\right)I(1,1)I^{2}(2,0)I(0,2),\\
		&C(f,g)=I^{2}(0,2)\left(I^{2}(2,0)I(1,2)+I^{2}(2,0)I(2,1)+I^{2}(1,1)I(3,0)\right),\\
		& D(f,g)=I(1,1)I(2,0)I^{2}(0,2)I\left(5/2,1/2\right),\ E(f,g)=I^{3}(2,0)I^{3}(0,2).
	\end{align*}
\end{theorem}
\begin{proof}
	To prove Theorem \ref{theo1}, we apply results from the theory of convergence of functional of stochastic process in \cite{cremers1986weak}. To this end, note that by~(\ref{eta}) and the definitions of $f_{n}$ and $g_{n}$, we write
	\begin{align*}
		\langle f_{n},g_{n} \rangle&=\int_{\mathcal{R}_{a}}\bigg(f(x)+\dfrac{\sqrt{k_{02}f(x)}}{\sqrt{nh_{n}}}\eta_{1n}(x)\bigg)\bigg(g(x)+\dfrac{\sqrt{k_{02}g(x)}}{\sqrt{nh_{n}}}\eta_{2n}(x)\bigg)dx,\\
		&=\int_{\mathcal{R}_{a}}f(x)g(x)dx+\dfrac{\sqrt{k_{02}}}{\sqrt{nh_{n}}}\int_{\mathcal{R}_{a}}f(x)\sqrt{g(x)}\eta_{2n}(x)dx\\
		&+\dfrac{\sqrt{k_{02}}}{\sqrt{nh_{n}}}\int_{\mathcal{R}_{a}}g(x)\sqrt{f(x)}\eta_{1n}(x)dx+\dfrac{k_{02}}{nh_{n}}\int_{\mathcal{R}_{a}}\sqrt{f(x)g(x)}\eta_{1n}(x)\eta_{2n}(x)dx.
	\end{align*}
	By (\ref{eta}) and Lemma \ref{lem2}, we can write $\langle f_{n},g_{n} \rangle$ as
	\[
	\sqrt{nh_{n}}(\langle f_{n},g_{n} \rangle-\langle f,g\rangle)=\sqrt{k_{02}}V_{n}^{(1)}+e_{n}^{(1)},
	\]
	where
	\[
	V_{n}^{(1)}=\int_{\mathcal{R}_{a}}(g(x)\sqrt{f(x)}\eta_{1n}(x)+f(x)\sqrt{g(x)}\eta_{2n}(x))dx.
	\]
	
	Similarly, we can calculate $\Vert f_{n}\Vert^{2}$ as follows 
	\[
	\Vert f_{n}\Vert^{2}=\int_{\mathcal{R}_{a}}\bigg(f(x)+\dfrac{\sqrt{k_{02}f(x)}}{\sqrt{nh_{n}}}\eta_{1n}(x)\bigg)^{2}dx.
	\]
	Therefore,
	\[
	\sqrt{nh_{n}}(\Vert f_{n}\Vert^{2}-\Vert f\Vert^{2})=2\sqrt{k_{02}}V_{n}^{(2)}+e_{n}^{(2)},
	\]
	where
	\[
	V_{n}^{(2)}=\int_{\mathcal{R}_{a}}f^{\frac{3}{2}}(x)\eta_{1n}(x)dx.
	\]
	
	Also, one can write 
	\[
	\sqrt{nh_{n}}(\Vert g_{n}\Vert^{2}-\Vert g\Vert^{2})=2\sqrt{k_{02}}V_{n}^{(3)}+e_{n}^{(3)},
	\]
	where
	\[
	V_{n}^{(3)}=\int_{\mathcal{R}_{a}}g^{\frac{3}{2}}(x)\eta_{2n}(x)dx.
	\]
	
	Hence,
	\begin{align}\label{Lim-dist}
		\sqrt{nh_{n}}\left(\begin{pmatrix}
			\langle f_{n},g_{n} \rangle\\
			\Vert g_{n}\Vert^{2}\\
			\Vert f_{n}\Vert^{2}
		\end{pmatrix}-
		\begin{pmatrix}
			\langle f,g\rangle\\
			\Vert g\Vert^{2}\\
			\Vert f\Vert^{2}
		\end{pmatrix}\right)&=\sqrt{k_{02}}\begin{pmatrix}
			V_{n}^{(1)}\\
			2V_{n}^{(2)}\\
			2V_{n}^{(3)}
		\end{pmatrix}+
		\begin{pmatrix}
			e_{n}^{(1)}\\
			e_{n}^{(2)}\\
			e_{n}^{(3)}
		\end{pmatrix},\notag\\
		&=:\sqrt{k_{02}}\bm V_{n}+\bm e_{n}.
	\end{align}
	
	Notice that by Lemma \ref{lem2}, $\bm e_{n}\xrightarrow{P} \bm0$ when $n\to\infty$. Now, we show that when $n\to\infty$, $\bm V_{n}\xrightarrow{w}\bm V: N(\bm0,k_{02}\bm\Sigma)$, where $\bm\Sigma$ is the covariance matrix of the random vector $\bm V=\left(V^{(1)},2V^{(2)},2V^{(3)}\right)^{T}$ with
	\[
	V^{(1)}=\int_{\mathcal{R}_{a}}(g(x)\sqrt{f(x)}Z_{1}(x)+f(x)\sqrt{g(x)}Z_{2}(x))dx,
	\]
	\[
	V^{(2)}=\int_{\mathcal{R}_{a}}f^{\frac{3}{2}}(x)Z_{1}(x)dx\ \text{and}\ V^{(3)}=\int_{\mathcal{R}_{a}}g^{\frac{3}{2}}(x)Z_{2}(x)dx,
	\]
	where $Z_{1}(x)$ and $Z_{2}(x)$ are Gaussian stochastic processes with zero mean and a covariance function given by $\Cov(Z_{i}(x),Z_{i}(y))=v_{x}(y)$, $i=1,2$, where $v_{x}(y)$ is the Dirac function which is defined as $v_{x}(y)=1$, if $x=y$, and $v_{x}(y)=0$, if $x\neq y$.
	In order to show that $\bm V_{n}\xrightarrow{d}\bm V$, we apply Theorem 3 in \cite{cremers1986weak}.
	Theorem~\ref{A} grantees that the finite dimensional distribution of $\xi_{n}:=\left(\eta_{1n}(x),\eta_{2n}(x)\right)^{T}$ converges to the finite dimensional distribution of $\xi_{0}(\cdot,x):=\left(Z_{1}(\cdot,x),Z_{2}(\cdot,x)\right)^{T}$.
	
	It remains to show that the sequence of the two-dimensional process $\xi_{n}$, $n=1,2,\dots$, satisfies the condition
	\begin{equation}\label{inf}
		\inf_{N>0}\limsup_{n\to\infty}\mathbb{P}\left(\Vert\xi_{n}(\omega,\cdot)\Vert_{2}>N\right)=0.
	\end{equation}
	
	To this end, we show that $\bm a^{T}\bm V_{n}\xrightarrow{d}\bm a^{T}\bm V$ for $\bm a\in\mathbb{R}^{3}\setminus\{\bm0\}$.
	
	The term $V_{n}^{(1)}$ can be written as follows:
	\begin{align*}
		V_{n}^{(1)}&=\int_{\mathcal{R}_{a}}(g(x)\sqrt{f(x)}\eta_{1n}(x)+f(x)\sqrt{g(x)}\eta_{2n}(x))dx,\\
		&=\int_{\mathcal{R}_{a}}\langle \xi_{n}(\cdot,x),\omega(x)\rangle_{*} dx,
	\end{align*}
	where $\langle s,s^{\prime}\rangle_{*}$ is the image of $s$ under $s^{\prime}$, $\xi_{n}(\cdot,x)=\left(\eta_{1n}(\cdot,x),\eta_{2n}(\cdot,x)\right)^{T}$ and  $\omega(w)$ belongs to the daul space of $L_{2}(\mathcal{R}_{a})$ such that
	$\omega(x)(v_{1},v_{2})=g(x)\sqrt{f(x)}v_{1}+f(x)\sqrt{g(x)}v_{2}$. Note that by Assumption \ref{ass1}\ref{ass_1}, $f$ and $g$ are bounded functions. By applying Theorem 3 in \cite{cremers1986weak} with $E=\mathbb{R}^{2}$, $p=q=2$ and $\omega(\cdot)$ in the corresponding daul space, we find 
	\[
	V_{n}^{(1)}\xrightarrow{d} \int_{\mathcal{R}_{a}}\langle \xi_{0}(\cdot,x),\omega(x)\rangle_{*} dx,\quad n\to\infty.
	\]
	
	Hence,
	\[
	V_{n}^{(1)}\xrightarrow{d}\int_{\mathcal{R}_{a}}(g(x)\sqrt{f(x)}Z_{1}(x)+f(x)\sqrt{g(x)}Z_{2}(x))dx.
	\]
	The terms $V_{n}^{(2)}$ and $V_{n}^{(3)}$ can be treated similarly. So, we get
	\begin{align*}
		\bm a^{T}\bm V_{n}&=a_{1}V_{n}^{(1)}+2a_{2}V_{n}^{(2)}+2a_{3}V_{n}^{(3)}\\
		&=\int_{\mathcal{R}_{a}}\left(a_{1}\sqrt{f(x)}g(x)+2a_{2}f^{3/2}(x)\right)\eta_{1n}(x)dx\\
		&+\int_{\mathcal{R}_{a}}\left(a_{1}\sqrt{g(x)}f(x)+2a_{2}g^{3/2}(x)\right)\eta_{2n}(x)dx\\
		&\xrightarrow{d} \bm a^{T}\bm V=\bm a^{T} \left(V^{(1)},2V^{(2)},2V^{(3)}\right)^{T}.
	\end{align*}
Since $\bm a$ is an arbitrary then by Cramér-Wold device, we get $\bm V_{n}\xrightarrow{d}\bm V$, as $n\to\infty$.
	
	To verify the condition (\ref{inf}), we apply the Markov's inequality, which gives 
	\begin{align*}
		\mathbb{P}\left(\Vert\left(\eta_{1n}(x),\eta_{2n}(x)\right)^{T}\Vert_{2}>N\right)&\leq N^{-2}\mathbb{E}\left(\Vert\left(\eta_{1n}(x),\eta_{2n}(x)\right)^{T}\Vert_{2}^{2}>N^{2}\right),\\
		&=N^{-2}\mathbb{E}\left(\int_{\mathcal{R}_{a}}(\eta_{1n}^{2}(x)+\eta_{2n}^{2}(x))^{\frac{1}{2}}dx\right),\\
		&\leq N^{-2}\int_{\mathcal{R}_{a}}(\mathbb{E}\eta_{1n}^{2}(x)+\mathbb{E}\eta_{2n}^{2}(x))^{\frac{1}{2}}dx.
	\end{align*}
		
	By Lemma \ref{NA}, $\int_{\mathcal{R}_{a}}(\mathbb{E}\eta_{1n}^{2}(x)+\mathbb{E}\eta_{2n}^{2}(x))^{\frac{1}{2}}dx<\infty$, for all $n$. 
	
	Using (\ref{bound-f}) and (\ref{bound-g}) and incorporating Fatous's lemma for $\limsup$, we get
	\begin{align*}
		&\limsup_{n\to\infty}\int_{\mathcal{R}_{a}}\sqrt{\mathbb{E}\eta_{1n}^{2}(x)+\mathbb{E}\eta_{2n}^{2}(x)}dx\leq\int_{\mathcal{R}_{a}}\sqrt{\limsup_{n\to\infty}\mathbb{E}\eta_{1n}^{2}(x)+\limsup_{n\to\infty}\mathbb{E}\eta_{2n}^{2}(x)}dx\\
		&\leq \int_{\mathcal{R}_{a}}\left[\limsup_{n\to\infty}\dfrac{2nh_{n}^{5}}{k_{02}f(x)}\left(\dfrac{B_{n,f}(x)}{h_{n}^{2}}\right)^{2}+\limsup_{n\to\infty}\dfrac{2nh_{n}^{5}}{k_{02}g(x)}\left(\dfrac{B_{n,g}(x)}{h_{n}^{2}}\right)^{2}\right]^{\frac{1}{2}}dx\\
		&\leq \int_{\mathcal{R}_{a}}\sqrt{C_{1}}dx=2a\sqrt{C_{1}}<\infty.
	\end{align*}
	Hence, for all $n$, we have
	\[
	\mathbb{E}\left(\int_{\mathcal{R}_{a}}\sqrt{\eta_{1n}^{2}(x)+\eta_{2n}^{2}(x)}dx\right)\leq C<\infty.
	\]
	By combining the above results for $\limsup\limits_{n\to\infty}P\left(\Vert\left(\eta_{1n}(x),\eta_{2n}(x)\right)^{T}\Vert_{2}>N\right)$, we~get
	\[
	\inf_{N>0}\limsup_{n\to\infty}P\left(\Vert\left(\eta_{1n}(x),\eta_{2n}(x)\right)^{T}\Vert_{2}>N\right)\leq\inf_{N>0}\dfrac{C}{N^{2}}=0.
	\]
	
	In the next step, we need $\bm\Sigma$, the covariance matrix of $\bm V$ to apply the $\delta$-method in order to derive the asymptotic distribution of $\rho(f_n,g_n)$.

Notice that as $Z_{1}(x)$ and $Z_{2}(x)$ are two independent Gaussian processes, then for $r,s\geq0$, it holds that
	\[
	\mathbb{E}\left(\int_{\mathcal{R}_{a}}f^{r}(x)g^{s}(x)Z_{j}(x)dx\right)=\int_{\mathcal{R}_{a}}f^{r}(x)g^{s}(x)\mathbb{E}Z_{j}(x)dx=0, \quad j=1,2.
	\]
	Hence,
	$\mathbb{E}(V^{(i)})=0$ for $i=1,2,3$. Further,
	\begin{align}\label{Eq_v}
		\Var(V^{(1)})&=\int_{\mathcal{R}_{a}}\int_{\mathcal{R}_{a}}\mathbb{E}\left(f(x)f(y)\sqrt{g(x)}\sqrt{g(y)}Z_{1}(x)Z_{1}(y)\right)dxdy\notag\\
		&+\int_{\mathcal{R}_{a}}\int_{\mathcal{R}_{a}}\mathbb{E}\left(g(x)g(y)\sqrt{f(x)}\sqrt{f(y)}Z_{2}(x)Z_{2}(y)\right)dxdy\notag,\\
		&=\Var\left(\int_{\mathcal{R}_{a}}g(x)\sqrt{f(x)}Z_{1}(x)dx\right)+Var\left(\int_{\mathcal{R}_{a}}f(x)\sqrt{g(x)}Z_{2}(x)dx\right)\notag\\
		&:=\Var(L_{1})+\Var(L_{2}).
	\end{align} 
	By applying the property $\int_{\mathcal{R}_{a}}v_{y}(x)h(x)dx=h(y)$ of Dirac function (see \cite{arkfenmathematical}) to the first part in (\ref{Eq_v}), we get
	\begin{align*}
		\Var(L_{1})&=\mathbb{E}\left(\int_{\mathcal{R}_{a}}\int_{\mathcal{R}_{a}}g(x)g(y)\sqrt{f(x)}\sqrt{f(y)}Z_{1}(x)Z_{1}(y)dxdy\right),\\
		&=\int_{\mathcal{R}_{a}}\int_{\mathcal{R}_{a}}v_{x}(y)\sqrt{f(x)}g(x)\sqrt{f(y)}g(y)dxdy,\\
		&=\int_{\mathcal{R}_{a}}f(x)g^{2}(x)dx=I(1,2).
	\end{align*}
	Similarly, for the second part in (\ref{Eq_v}), we get
	\[
	\Var(L_{2})=\int_{\mathcal{R}_{a}}f^{2}(x)g(x)dx=I(2,1).
	\]
	Hence, (\ref{Eq_v}) becomes $\Var(V^{(1)})=I(1,2)+I(2,1)$, which is finite as by Assumption \ref{ass1} both integrals $I(1,2)$ and $I(1,2)$ are finite.
	
	Similarly, by Assumption \ref{ass1} we obtain
	\[
	\Var(V^{(2)})=\int_{\mathcal{R}_{a}}f^{3}(x)dx=I(3,0)<\infty
	\]
	and
	\[
	\Var(V^{(3)})=\int_{\mathcal{R}_{a}}g^{3}(x)dx=I(0,3)<\infty.
	\]
	Also, it can be shown that
	\[
	\Cov(V^{(1)},V^{(2)})=\int_{\mathcal{R}_{a}}f^{\frac{5}{2}}(x)g^{\frac{1}{2}}(x)dx=I(5/2,1/2)<\infty,
	\]
	\[
	\Cov(V^{(1)},V^{(3)})=\int_{\mathcal{R}_{a}}f^{\frac{1}{2}}(x)g^{\frac{5}{2}}(x)dx=I\left(1/2,5/2\right)<\infty,
	\]
	and $\Cov(V^{(2)},V^{(3)})=0$.
	
	Using the above results one can write the matrix $\bm\Sigma$ as
	\[
	\bm\Sigma=\begin{pmatrix}
		I(2,1)+I(1,2) & 2I(\frac{5}{2},\frac{1}{2}) & 2I(\frac{1}{2},\frac{5}{2})\\
		2I(\frac{5}{2},\frac{1}{2}) & 4I(3,0) & 0\\
		2I(\frac{1}{2},\frac{5}{2}) & 0 & 4I(0,3)
	\end{pmatrix}.
	\]
	To complete the proof of the main theorem, we apply the $\delta$-method. To this end, we define the function $\psi:\mathbb{R}^{+}\times\mathbb{R}^{+}\times\mathbb{R}^{+}\to \mathbb{R}^{+}$, by $\psi(t_{1},t_{2},t_{3})=t_{1}/\sqrt{t_{2}t_{3}}$. The gradient of $\psi(\cdot,\cdot,\cdot)$ is 
	\[
	\nabla\psi=\left(t_{2}^{-\frac{1}{2}}t_{3}^{-\frac{1}{2}},-\frac{1}{2}t_{1}t_{2}^{-\frac{3}{2}}t_{3}^{-\frac{1}{2}},-\frac{1}{2}t_{1}t_{2}^{-\frac{1}{2}}t_{3}^{-\frac{3}{2}}\right)^{T}.
	\]
	Therefore, 
	\[
	\nabla\psi^{T}\bm\Sigma\nabla\psi=\dfrac{A(f,g)-2B(f,g)+C(f,g)-2D(f,g)}{E(f,g)},
	\]
	with
	\begin{align*}
		A(f,g)&=I(0,3)I^{2}(1,1)I^{2}(2,0),\ B(f,g)=I\left(1/2,5/2\right)I(1,1)I^{2}(2,0)I(0,2),\\
		&C(f,g)=I^{2}(0,2)\left(I^{2}(2,0)I(1,2)+I^{2}(2,0)I(2,1)+I^{2}(1,1)I(3,0)\right),\\
		& D(f,g)=I(1,1)I(2,0)I^{2}(0,2)I\left(5/2,1/2\right),\ E(f,g)=I^{3}(2,0)I^{3}(0,2).
	\end{align*}
	
	Now, the asymptotic distribution of $\sqrt{nh_{n}}(\rho(f_{n},g_{n})-\rho(f,g))$ is obtained by applying the $\delta$-method to (\ref{Lim-dist}) with $\psi(\cdot)$
	\[
	\sqrt{nh_{n}}\left(\psi(\langle f_{n},g_{n} \rangle,\Vert f_{n}\Vert^{2},\Vert g_{n}\Vert^{2})- \psi(\langle f,g \rangle,\Vert f\Vert^{2},\Vert g\Vert^{2})\right)\xrightarrow{d}N\left(0,k_{02}\nabla\psi^{T}\bm\Sigma\nabla\psi\right).
	\]
	Hence,
	\[
	\sqrt{nh_{n}}(\rho(f_{n},g_{n})-\rho(f,g))\xrightarrow{d} N\left(0,\sigma_{f,g}^{2}\right),
	\]
	which completes the proof.
	\end{proof}

Note that by Lemma {\rm\ref{lem1}} and Theorem~\ref{theo1} we can prove the following result.
\begin{colo}\label{col1}
	Let $f_{n}(\cdot)$ and $g_{n}(\cdot)$ be kernel estimators of $f$ and $g$, respectively, with the kernel $K(\cdot)$. If Assumptions {\rm\ref{ass1}}-{\rm\ref{ass4}} hold, then, when $n\to\infty$, 
	\[
	\dfrac{\sqrt{nh_{n}}(\rho(f_{n},g_{n})-\rho(f,g))}{\sigma_{f_n,g_n}}\xrightarrow{d} N\left(0,1\right),
	\]
	where 
	\[
	\sigma_{f_n,g_n}^{2}:=\dfrac{k_{02}(A(f_{n},g_{n})-2B(f_{n},g_{n})+C(f_{n},g_{n})-2D(f_{n},g_{n}))}{E(f_{n},g_{n})}
	\]
	with
	\begin{align*}
		A(f_{n},g_{n})&=I_{n}(0,3)I_{n}^{2}(1,1)I_{n}^{2}(2,0),\ B(f_{n},g_{n})=I_{n}\left(1/2,5/2\right)I_{n}(1,1)I_{n}^{2}(2,0)I_{n}(0,2),\\
		&C(f_{n},g_{n})=I_{n}^{2}(0,2)\left(I_{n}^{2}(2,0)I_{n}(1,2)+I_{n}^{2}(2,0)I_{n}(2,1)+I_{n}^{2}(1,1)I_{n}(3,0)\right),\\
		& D(f_{n},g_{n})=I_{n}(1,1)I_{n}(2,0)I_{n}^{2}(0,2)I_{n}\left(5/2,1/2\right),\ E(f,g)=I_{n}^{3}(2,0)I_{n}^{3}(0,2).
	\end{align*}
\end{colo}
In some applications, an interval estimation for $\rho(f,g)$  is required. So, we use Corollary {\rm\ref{col1}} to construct a $100(1-\alpha)\%$ confidence interval for $\rho(f,g)$ as
\[
\rho(f_{n},g_{n})\mp z_{1-\frac{\alpha}{2}}\dfrac{\sigma_{f_n,g_n}}{\sqrt{nh_{n}}},
\]
where $z_{1-\frac{\alpha}{2}}$ is the $100(1-\frac{\alpha}{2})$ standard normal percentile. 
\section{Resuts for the MacArthur-Levins measure $\Delta_{n}$}\label{sec6}
In this section, we investigate the asymptotic behavior of the MacArthur-Levins overlapping measure $\Delta_{n}(f,g)$. The results of this section can be derived analogously to those in the previous section. To avoid repetition, we list only steps that need to be investigated.
\begin{theorem}\label{theo2}
	Let $f(\cdot)$ and $g(\cdot)$ be two probability density functions satisfying Assumption {\rm\ref{ass1}}. Let $K(\cdot)$ and $h_{n}$ satisfy Assumptions {\rm\ref{ass3}} and {\rm\ref{ass4}}, respectively. Then, when $n\to\infty$, we~have $\sqrt{nh_{n}}(\Delta(f_{n},g_{n})-\Delta(f,g))\xrightarrow{d} N( 0, k_{02}\sigma_{*}^{2}(f,g))$, where
	\begin{align*}
		\sigma_{*}^{2}(f,g)=\frac{I(2,1)+I(1,2)}{I^{2}(0,2)}-4\frac{I(1/2,5/2)I(1,1)}{I^{5}(0,2)}+4\frac{I(0,3)I^{2}(1,1)}{I^{8}(0,2)}.
	\end{align*}
\end{theorem}

To sketch the proof of Theorem \ref{theo2}, we define the matrix $\bm E$ as 
\[
\bm E=\begin{pmatrix}
	1 & 0 & 0\\
	0 & 0 & 1
\end{pmatrix}.
\]
Then, when $n\to\infty$ we have 
\begin{align*}
	\sqrt{nh_{n}}(\Delta(f_{n},g_{n})-\Delta(f,g))&=\sqrt{nh_{n}}\left(\bm E\begin{pmatrix}
		\langle f_{n},g_{n} \rangle\\
		\Vert g_{n}\Vert^{2}\\
		\Vert f_{n}\Vert^{2}
	\end{pmatrix}-
	\bm E\begin{pmatrix}
		\langle f,g\rangle\\
		\Vert g\Vert^{2}\\
		\Vert f\Vert^{2}
	\end{pmatrix}\right)\notag\\
&\xrightarrow{d}N(\bm 0,\bm E\bm\Sigma\bm E^{T}),
\end{align*}
with
\[
\bm E\bm\Sigma\bm E^{T}=\begin{pmatrix}
	I(2,1)+I(1,2) & 2I(1/2,5/2)\\
	2I(1/2,5/2) & 4I(0,3)
\end{pmatrix}.
\]
Let $\psi_{1}(t_{1},t_{2})=\frac{t_{1}}{t_{2}}$, then $\nabla \psi_{1}(t_{1},t_{2})=\left(\frac{1}{t_{2}},-\frac{t_{1}}{t_{2}^{2}}\right)^{T}$. By applying the $\delta$-method, we obtain 
\[
\sqrt{nh_{n}}(\Delta(f_{n},g_{n})-\Delta(f,g))\xrightarrow{d} N( 0, k_{02}\sigma_{*}^{2}(f,g)),\qquad n\to\infty,
\]
where 
\begin{align*}
	\sigma_{*}^{2}(f,g)&=\nabla \psi_{1}^{T}(\langle f,g\rangle,\Vert g\Vert^{2})\begin{pmatrix}
		I(2,1)+I(1,2) & 2I(1/2,5/2)\\
		2I(1/2,5/2) & 4I(0,3)
	\end{pmatrix}\nabla \psi_{1}(\langle f,g\rangle,\Vert g\Vert^{2}),\\
	&=\left(\dfrac{1}{I(0,2)},-\dfrac{I(1,1)}{I^{4}(0,2)}\right)\begin{pmatrix}	I(2,1)+I(1,2) & 2I(1/2,5/2)\\
		2I(1/2,5/2) & 4I(0,3)
	\end{pmatrix}\begin{pmatrix}
		\dfrac{1}{I(0,2)}\\
		-\dfrac{I(1,1)}{I^{4}(0,2)}
	\end{pmatrix},\\
	&= \frac{I(2,1)+I(1,2)}{I^{2}(0,2)}-4\frac{I(1/2,5/2)I(1,1)}{I^{5}(0,2)}+4\frac{I(0,3)I^{2}(1,1)}{I^{8}(0,2)}.
\end{align*}
\begin{colo}\label{col2}
	Let $f_{n}(\cdot)$ and $g_{n}(\cdot)$ be kernel estimators of $f$ and $g$, respectively, with the kernel $K(\cdot)$. If Assumptions {\rm\ref{ass1}}-{\rm\ref{ass4}} hold, then, when $n\to\infty$, 
	\[
	\dfrac{\sqrt{nh_{n}}(\Delta(f_{n},g_{n})-\Delta(f,g))}{\sqrt{k_{02}}\sigma_{*}(f_{n},g_{n})}\xrightarrow{d} N\left(0,1\right),
	\]
	where
	\begin{align*}
		\sigma_{*}^{2}(f_{n},g_{n})=\frac{I_{n}(2,1)+I_{n}(1,2)}{I_{n}^{2}(0,2)}-4\frac{I_{n}(1/2,5/2)I_{n}(1,1)}{I_{n}^{5}(0,2)}+4\frac{I_{n}(0,3)I_{n}^{2}(1,1)}{I_{n}^{8}(0,2)}.
	\end{align*}
\end{colo}

Also, a $100(1-\alpha)\%$ confidence interval for $\Delta(f,g)$ can be found as
\[
\Delta(f_{n},g_{n})\mp z_{1-\frac{\alpha}{2}}\dfrac{\sqrt{k_{02}}\sigma_{*}(f_{n},g_{n})}{\sqrt{nh_{n}}}.
\]

\section{Numerical Studies}\label{sec7}
This section presents a numerical study to investigate the asymptotic behaviour of $\rho(f_n,g_n)$. For simulation, we used truncated versions of two known density functions, namely, the normal and the logistic density functions. The numerical results are illustrated via the following two cases.
\begin{description}
		\item[Case I] Both $f$ and $g$ are normal density functions
	\end{description}
For this case, 
we assumed that $f$ and $g$ are the pdfs of $N(1,4^2)$ and $N(5,4.5^2)$ respectively. The support $\mathcal{R}_{a}=[-a,a]$ was constructed by defining $a$ as the $0.995$th quantile of $f$. 
The pdfs of the truncated versions $\bar{f}(x)$ and $\bar{g}(x)$ take the forms 
$\bar{f}(x)=\frac{f(x)}{1-2F(-a)}$ and $\bar{g}(x)=\frac{g(x)}{1-2G(-a)}$, $x\in\mathcal{R}_{a}$.  
For the bandwidth, we used the sequence $h_{n}=\frac{\sqrt{\log(n)}}{0.45n^{\frac{2}{3}}}$. The density functions $f$ and $g$ were estimated using $f_{n}$ and $g_{n}$, respectively, for samples of sizes $50,\ 150$ and $500$. Then, the simulated samples were used to compute $\hat{\rho}_{n}:=\sqrt{nh_{n}}(\rho_{n}(f_{n},g_{n})-\rho(f,g)),\ n=50,\ 150$ and $500$. For each $n$, the histogram of $\hat{\rho}_{n}$ was constructed. In order to compare the simulation results with the asymptotic distribution in Theorem \ref{theo1}, we plotted the density of $N(0,\sigma_{f,g}^{2})$ along with the corresponding histogram 
(see Figures (\ref{fig1:case1a})-(\ref{fig1:case3a})).
Further, to compare empirical distributions, a Q-Q plot of $\hat{\rho}_{n}$ versus the quantile of $N(0,\sigma_{f,g}^{2})$ for each sample was produced (see Figures (\ref{fig1:case1b})-(\ref{fig1:case3b})). From these figures it is clear that when $n$ becomes large these empirical distributions get closer to the corresponding asymptotic distributions, which supports the theoretical findings.
\vspace{-.4cm} 
\begin{figure}[H]
	\centering
	\begin{subfigure}[b]{0.4\textwidth}
		\includegraphics[width=1.25\textwidth,height=5.1cm,
		trim={0cm 0 0 0},clip]{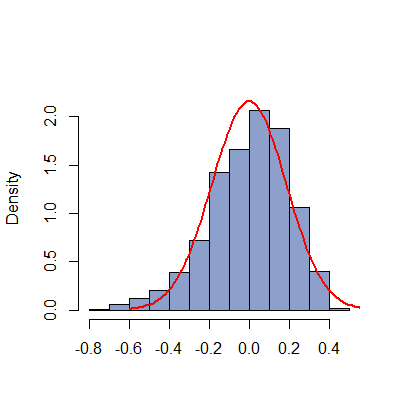}  \vspace{-0.2cm}
		\caption{}
		\label{fig1:case1a}
	\end{subfigure}
	\hspace{1cm}
	\begin{subfigure}[b]{0.4\textwidth}
		\includegraphics[width=1.25\textwidth, height=5.1cm,trim={0cm 0 0 0},clip]{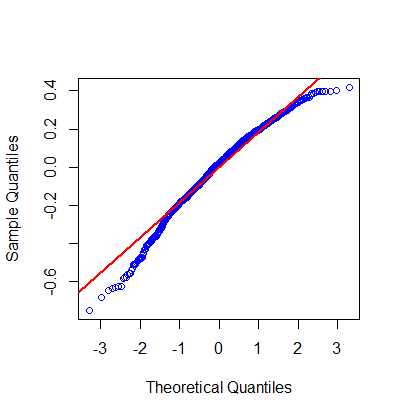} \vspace{-0.2cm}
		\caption{}
		\label{fig1:case1b}
	\end{subfigure} \hspace{2cm}
	\vspace{-.2cm}
	\caption{(a) The histogram of $\hat{\rho}_{n}$, with $n=50$, (b) A Q-Q plot of $\hat{\rho}_{n}$ versus the quantile of $N(0,\sigma_{f,g}^{2})$ with $n=50$.} \label{fig1}
\end{figure}
\vspace{-.4cm} 
\begin{figure}[H]
	\centering
	\begin{subfigure}[b]{0.4\textwidth}
		\includegraphics[width=1.25\textwidth,height=5.1cm,
		trim={0cm 0 0 0},clip]{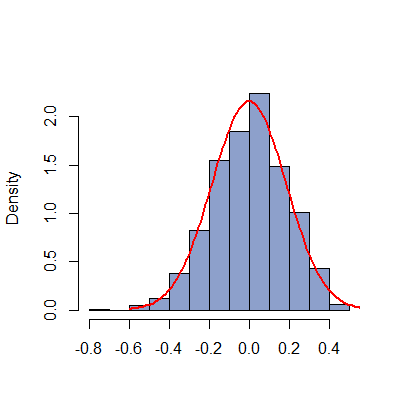}  \vspace{-0.2cm}
		\caption{}
		\label{fig1:case2a}
	\end{subfigure}
	\hspace{1cm}
	\begin{subfigure}[b]{0.4\textwidth}
		\includegraphics[width=1.25\textwidth, height=5.1cm,trim={0cm 0 0 0},clip]{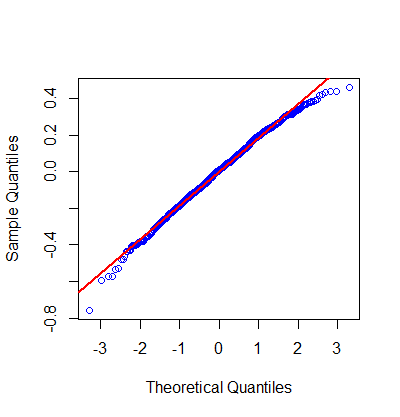} \vspace{-0.2cm}
		\caption{}
		\label{fig1:case2b}
	\end{subfigure} \hspace{2cm}
	\vspace{-.4cm}
	\caption{(a) The histogram of $\hat{\rho}_{n}$, with $n=150$, (b) A Q-Q plot of $\hat{\rho}_{n}$ versus the quantile of $N(0,\sigma_{f,g}^{2})$ with $n=150$.} \label{fig2}
\end{figure}
\vspace{-.4cm} 
\begin{figure}[H]
	\centering
	\begin{subfigure}[b]{0.4\textwidth}
		\includegraphics[width=1.25\textwidth,height=5.1cm,
		trim={0cm 0 0 0},clip]{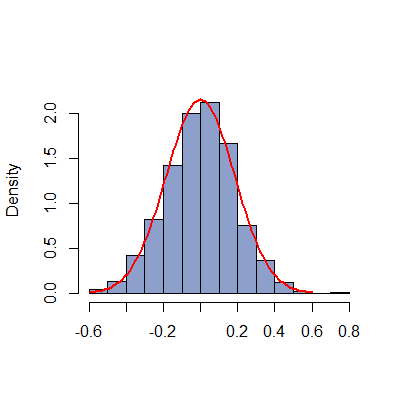}  \vspace{-0.2cm}
		\caption{}
		\label{fig1:case3a}
	\end{subfigure}
	\hspace{1cm}
	\begin{subfigure}[b]{0.4\textwidth}
		\includegraphics[width=1.25\textwidth, height=5.1cm,trim={0cm 0 0 0},clip]{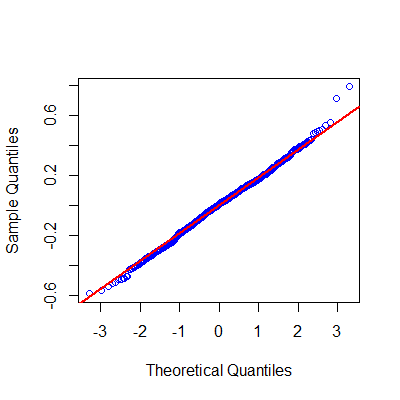} \vspace{-0.2cm}
		\caption{}
		\label{fig1:case3b}
	\end{subfigure} \hspace{2cm}
	\vspace{-.4cm}
	\caption{(a) The histogram of $\hat{\rho}_{n}$, with $n=500$, (b) A Q-Q plot of $\hat{\rho}_{n}$ versus the quantile of $N(0,\sigma_{f,g}^{2})$ with $n=500$.} \label{fig3}
\end{figure}

\begin{description}
	\item[Case II] $f$ and $g$ are the normal and the logistic density functions, respectively.
\end{description}
Here we consider $\bar{f}(x)$ and $\bar{g}(x)$ as the truncated versions of the pdfs of $N(5,4^2)$ and $Logistic(0,3)$, respectively. Also, $h_{n}$ and sample sizes are remained as in the previous case. 
Figures (\ref{fig2:case1a})-(\ref{fig6:case1a}) and Figures (\ref{fig2:case1b})-(\ref{fig6:case1b}) show the histograms and the corresponding Q-Q plots of $\hat{\rho}_{n}$ versus the quantile of $N(0,\sigma_{f,g}^{2})$ for each $n=50,\ 150$ and $500$. From these figures we also conclude that the empirical distributions get closer to the corresponding asymptotic distributions when $n$ becomes large, which supports the theoretical findings.
\vspace{-.4cm} 
\begin{figure}[H]
	\centering
	\begin{subfigure}[b]{0.4\textwidth}
		\includegraphics[width=1.25\textwidth,height=5.1cm,
		trim={0cm 0 0 0},clip]{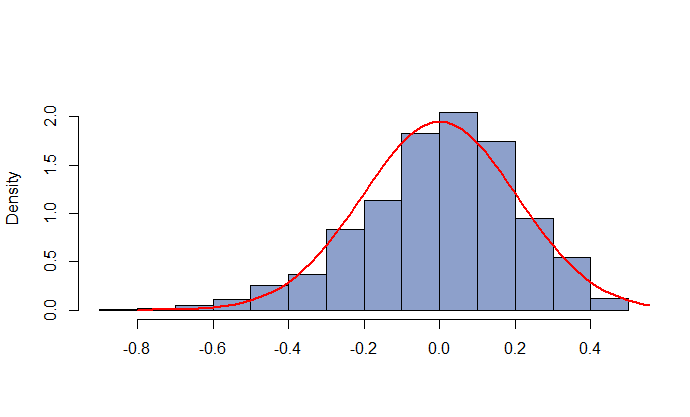}  \vspace{-0.2cm}
		\caption{}
		\label{fig2:case1a}
	\end{subfigure}
	\hspace{1cm}
	\begin{subfigure}[b]{0.4\textwidth}
		\includegraphics[width=1.25\textwidth, height=5.1cm,trim={0cm 0 0 0},clip]{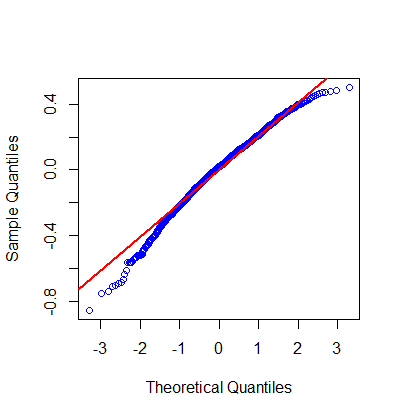} \vspace{-0.2cm}
		\caption{}
		\label{fig2:case1b}
	\end{subfigure} \hspace{2cm}
	\vspace{-.4cm}
	\caption{(a) The histogram of $\hat{\rho}_{n}$, with $n=50$, (b) A Q-Q plot of $\hat{\rho}_{n}$ versus the quantile of $N(0,\sigma_{f,g}^{2})$ with $n=50$.} \label{fig4}
\end{figure}
\vspace{-.5cm} 
\begin{figure}[H]
	\centering
	\begin{subfigure}[b]{0.4\textwidth}
		\includegraphics[width=1.25\textwidth,height=5.1cm,
		trim={0cm 0 0 0},clip]{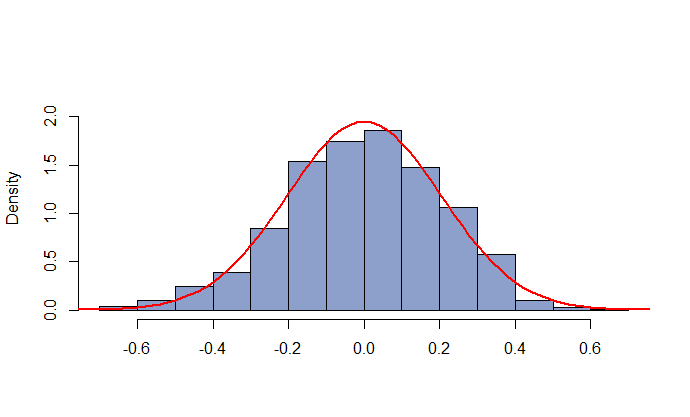}  \vspace{-0.2cm}
		\caption{}
		\label{fig5:case1a}
	\end{subfigure}
	\hspace{1cm}
	\begin{subfigure}[b]{0.4\textwidth}
		\includegraphics[width=1.25\textwidth, height=5.1cm,trim={0cm 0 0 0},clip]{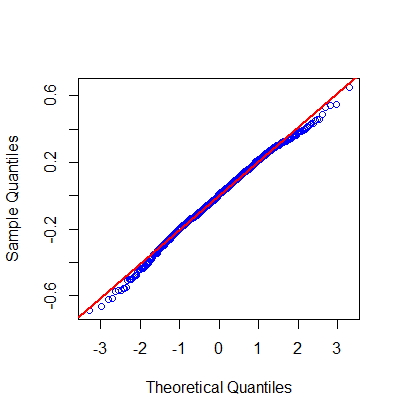} \vspace{-0.2cm}
		\caption{}
		\label{fig5:case1b}
	\end{subfigure} \hspace{2cm}
	\vspace{-.2cm}
	\caption{(a) The histogram of $\hat{\rho}_{n}$, with $n=150$, (b) A Q-Q plot of $\hat{\rho}_{n}$ versus the quantile of $N(0,\sigma_{f,g}^{2})$ with $n=150$.} \label{fig5}
\end{figure}
\vspace{-.5cm} 
\begin{figure}[H]
	\centering
	\begin{subfigure}[b]{0.4\textwidth}
		\includegraphics[width=1.25\textwidth,height=5.1cm,
		trim={0cm 0 0 0},clip]{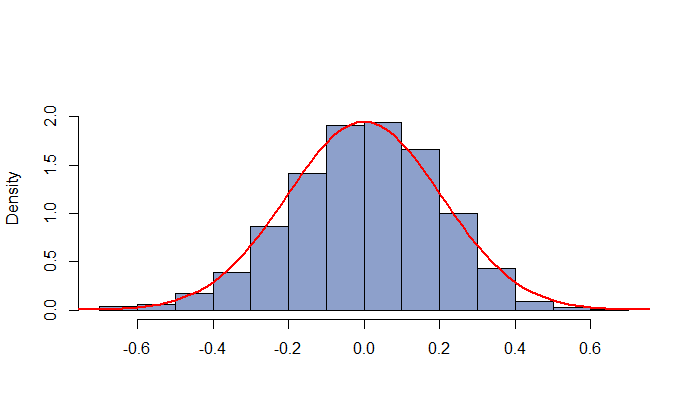}  \vspace{-0.1cm}
		\caption{}
		\label{fig6:case1a}
	\end{subfigure}
	\hspace{1cm}
	\begin{subfigure}[b]{0.4\textwidth}
		\includegraphics[width=1.25\textwidth, height=5.1cm,trim={0cm 0 0 0},clip]{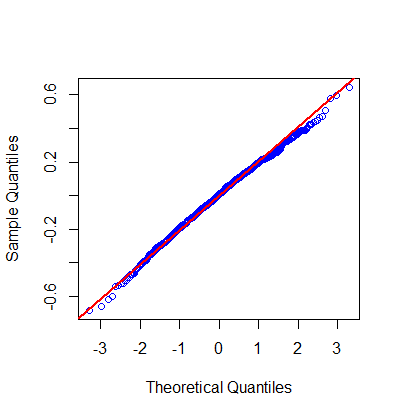} \vspace{-0.2cm}
		\caption{}
		\label{fig6:case1b}
	\end{subfigure} \hspace{2cm}
	\vspace{-.2cm}
	\caption{(a) The histogram of $\hat{\rho}_{n}$, with $n=500$, (b) A Q-Q plot of $\hat{\rho}_{n}$ versus the quantile of $N(0,\sigma_{f,g}^{2})$ with $n=500$.} \label{fig6}
\end{figure}

\section{Application to the Breast Cancer Data}\label{app}
Based on global statistics, breast cancer is one of the most common cancers among women worldwide that cause death. The unusual growth of cells in the breast tissue forms tumors that might be benign (non-cancerous) or malignant (cancerous).
In this study, we use the Breast Cancer Wisconsin (Diagnostic) DataSet created by Dr. William H. Wolberg at the University of Wisconsin Hospital and published on Kaggle. The dataset was collected from fluid samples of patients with solid breast masses and using a software system called Xcyt.
The tumor was diagnosed as malignant (M) or benign (B) for each patient. Some features were also computed for each cell such as radius, texture, perimeter, area and smoothness. We are interested in studying the perimeter of malignant (M) and benign (B) tumors.

For illustration, we considered the case of equal sample sizes $n=212$ with the bandwidth $h_{n}=4.2/n^{\frac{2}{3}}$. 
Then, we estimated the density function via the kernel estimation method for the groups M and B. Figures (\ref{fig7:case1a}) and (\ref{fig7:case1b}) show the estimated kernel density functions for such groups along with their histograms.
\vspace{-.3cm} 
\begin{figure}[H]
	\centering
	\begin{subfigure}[b]{0.4\textwidth}
		\includegraphics[width=1.25\textwidth,height=5.1cm,
		trim={0cm 0 0 0},clip]{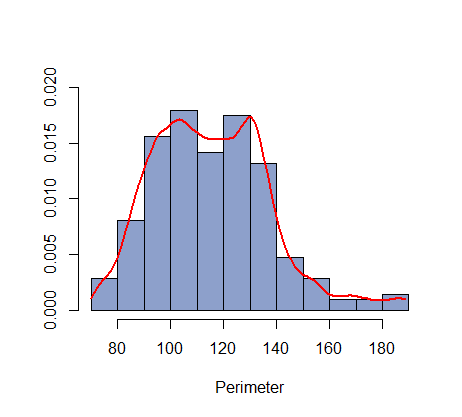}  \vspace{-0.1cm}
		\caption{}
		\label{fig7:case1a}
	\end{subfigure}
	\hspace{1cm}
	\begin{subfigure}[b]{0.4\textwidth}
		\includegraphics[width=1.25\textwidth, height=5.1cm,trim={0cm 0 0 0},clip]{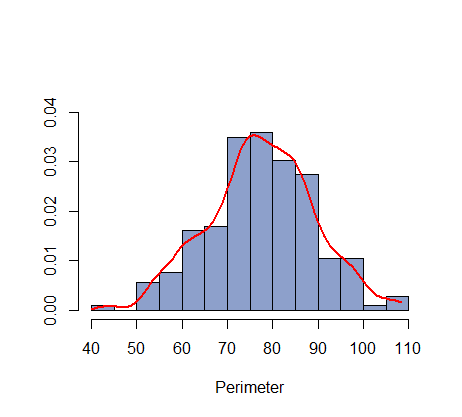} \vspace{-0.2cm}
		\caption{}
		\label{fig7:case1b}
	\end{subfigure} \hspace{2cm}
	\vspace{-.2cm}
	\caption{(a) The histogram of group M along with its estimated pdf, (b)  The histogram of group B along with its estimated pdf.} \label{fig7}
\end{figure}
In order to illustrate the overlap between the two groups graphically, we provided Figure (\ref{fig8}) which shows the overlapping area under the two estimated kernel density functions. Although this area is not the same as the value of $\rho_{n}$, it indicates that there is an overlapping between groups M and B.
\vspace{-.4cm}
\begin{figure}[H]
	\centering
	\includegraphics[width=0.65\textwidth, height=5.95cm,trim={0cm 0 0 0},clip]{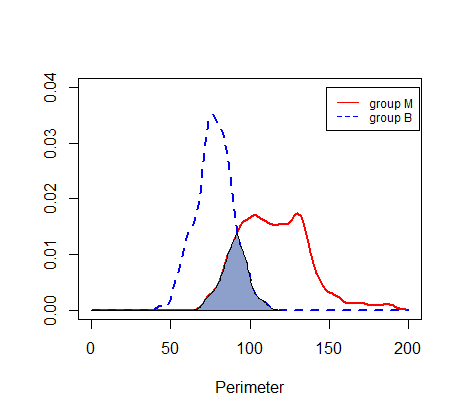} \vspace{-0.4cm}
	\caption{The overlapping area between the estimated kernel density functions for groups M and group B}
	\label{fig8}
\end{figure}
Based on our approach, the estimated overlapping measure is $\hat{\rho}_{n}=0.3396$ with a standard error of $0.0117$. Also, a $95\%$ confidence interval for $\rho$ is $[0.3203, 0.3588]$. Since this interval does not include $1$, we conclude that there is a significant difference between distributions of malignant (M) and benign (B) groups. This means that the perimeter variable can be used to separate the cases of being either malignant or benign.

\section*{Acknowledgements}
The authors would like to thank the anonymous referees for their suggestions that helped to improve the paper.

\bibliography{alodat_mybibfile}
\end{document}